\documentclass[a4paper,11pt]{amsart}
\usepackage[utf8]{inputenc}
\usepackage{amsmath}
\usepackage{amsthm}
\usepackage{thmtools}
\usepackage{amssymb}
\usepackage{thm-restate}
\usepackage{hyperref}
\usepackage[hmarginratio=1:1]{geometry}
\usepackage{enumitem}

\title{On the congruence subgroup problem for branch groups}
\author{Alejandra Garrido}
\thanks{This research was supported by \emph{Fundaci\'{o}n La Caixa, Spain}}
\thanks{I would like to thank my supervisor Prof. John S. Wilson for his help in the preparation of this paper and the referee for useful comments.}

\declaretheorem[name=Theorem]{mainthm}

\declaretheorem[name=Theorem, numberwithin=section]{thm}
\declaretheorem[name=Lemma,sibling=thm]{lem}
\declaretheorem[name=Proposition, sibling=thm]{prop}

\newcommand{\lf}{\leq_{\mathrm{f}}}
\newcommand{\nf}{\trianglelefteq_{\mathrm{f}}}
\newcommand{\lva}{\leq_{\mathrm{va}}}
\newcommand{\nva}{\trianglelefteq_{\mathrm{va}}}

\DeclareMathOperator{\St}{Stab}
\DeclareMathOperator{\rst}{rist}
\DeclareMathOperator{\Aut}{Aut}
\DeclareMathOperator{\R}{R_G}
\DeclareMathOperator{\N}{N_G}
\DeclareMathOperator{\C}{C_G}

\begin{document}

\begin{abstract}
We answer a question of Bartholdi, Siegenthaler and Zalesskii, 
showing that the congruence subgroup problem for branch groups is independent of the branch action on a tree.
We prove that the congruence topology of a branch group is determined by the group;
specifically, by its structure graph, an object first introduced by Wilson.
We also give a more natural definition of this graph.
\end{abstract}

\maketitle

\section{Introduction}

Groups acting on rooted trees have been the subject of intense study over the past few decades
after the appearance in the 1980s of examples with exotic properties 
(e.g.\ finitely generated infinite torsion groups, groups of intermediate word growth, amenable but not elementary amenable groups, etc.).
Several attempts were made at the time to round up these examples into one class of groups.
One of these led to the definition of branch groups (\cite{Branchgroups}),
which also arise in the classification of just infinite groups (\cite{wilsonNewhorizons}). 

For a sequence $(m_n)_{n\geq 0}$ of integers $m_n\geq 2$, the \emph{rooted tree of type $(m_n)$} is a tree $T$
with a distinguished vertex $v_0$, called the \emph{root}, of valency $m_0$ and such that every vertex at distance $n\geq 1$ from $v_0$ has
valency $m_n+1$
(where the distance of a vertex from $v_0$ is the number of edges in the unique path from that vertex to $v_0$).
The set of all vertices at distance $n$ from $v_0$ is the $n$th layer of $T$, denoted by $V_n$. 
We picture $T$ with $v_0$ at the top and with $m_n$ edges descending from each vertex in $V_n$, 
so we call the vertices below a given $v$ the \emph{descendants} of $v$.
Each vertex $v\in V_r$ is the root of a subtree $T_v$ of type $(m_n)_{n\geq r}$.

Let $G$ be a group acting faithfully on $T$ fixing $v_0$. 
For each vertex $v$, the \emph{rigid stabilizer of $v$} is the subgroup $\rst_G(v)$ of elements of $G$ which fix every vertex outside $T_v$.
For each $n\geq 0$, the direct product $\rst_G(n)=\langle \rst_G(v) \mid v \in V_n\rangle$ 
is the \emph{rigid stabilizer of the $n$th layer}.
We call the faithful action of $G$ on $T$ a \emph{branch action} if the following holds for all $n\geq 0$:
\begin{enumerate}[label=(\roman*)]
 \item $G$ acts transitively on $V_n$;
 \item $\rst_G(n)$ has finite index in $G$.
\end{enumerate}
We say that $G$ is a \emph{branch group} if there exists a branch action of $G$ on some tree $T$. 

Since branch groups have such specific actions on rooted trees, it is natural to wonder what the action tells us about the subgroup structure of the group.
Consider, for each $n\geq0$, the kernel $\St_G(n)$ of the action of $G$ on $V_n$;
can we ``detect'' every finite index subgroup of $G$ (or, equivalently, every finite quotient) by looking at the finite quotients $G/\St_G(n)$?
In other words, does every finite index subgroup of $G$ contain some $\St_G(n)$?
We can rephrase this question in terms of profinite completions. 
Taking the subgroups $\{\St_G(n) \mid n\geq0\}$ as a neighbourhood basis for the identity gives a topology on $G$ -- the \emph{congruence topology} --
and the completion $\overline{G}$ of $G$ with respect to this topology is a profinite group called the \emph{congruence completion} of $G$.
As $G$ acts faithfully on $T$ we have $\bigcap_n \St_G(n)=1$, so $G$ embeds in $\overline{G}$.
\emph{A fortiori}, $G$ is residually finite, so it also embeds in its profinite completion $\widehat{G}$ which maps onto $\overline{G}$. 
Asking whether each finite index subgroup of $G$ contains some stabilizer $\St_G(n)$ is tantamount to asking whether the map $\widehat{G}\to \overline{G}$ is injective.
The \emph{congruence subgroup problem} asks us to compute the \emph{congruence kernel} $C$ of this map, which measures the deviation from a positive answer.

Since a branch group $G$ has another obvious family of finite index subgroups, namely $\{\rst_G(n)\mid n\geq 0\}$, 
we may ask the same question for this family. 
Let $\widetilde{G}$ denote the \emph{branch completion} of $G$; 
that is, the completion of $G$ with respect to the \emph{branch topology}, 
which is generated by taking $\{\rst_G(n)\mid n\geq0\}$ as a neighbourhood basis of the identity.
Then, as above, there is a surjective homomorphism $\widehat{G}\rightarrow\widetilde{G}$
and we are asked to determine the \emph{branch kernel} $B$ of this map.
Note that the branch topology is stronger than the congruence topology 
so that we may also ask about the kernel $R$ of the map $\widetilde{G}\rightarrow\overline{G}$, the \emph{rigid kernel}.

The term ``congruence'' is used by analogy with the classical congruence subgroup problem for $\mathrm{SL}_n(\mathbb{Z})$ (solved in \cite{bass1964,mennicke}), from which these questions take inspiration.
There are now several generalizations of this problem:
 for instance, a now classical generalization in the context of algebraic groups
 (see \cite{RaghunathanSurvey} and references therein), 
 and a more recent one in the context of automorphisms of free groups $F_n$, 
with the kernels of the action on finite quotients of $F_n$ playing the role of our subgroups $\St_G(n)$
(see \cite{BuxErshovRapinchuk}).

The problem of determining the congruence, branch and rigid kernels for a branch group $G$ was first posed in \cite{bartholdiCSP}, 
where the authors also ask a ``preliminary question of great importance'':\\

\noindent\textbf{Question} Do any of the kernels depend on the branch action of $G$?\\

In other words, is the nature of these kernels a property of the branch action or of the group?
% In \cite{bartholdiCSP} this question is partially answered with the following theorems:

% \begin{thma}[\cite{bartholdiCSP}]
 %If a branch group has two self-similar, regular branch actions then their respective congruence kernels coincide.
%\end{thma}

%\begin{thmd}[\cite{bartholdiCSP}]
% Given two branch actions of a branch group, 
% the branch kernel with respect to one is contained in the congruence kernel with respect to the other.
%\end{thmd}

We provide a full answer to the above question by proving the following:
\begin{mainthm}\label{congruence}
 Let $G$ have two branch actions on trees. 
 Then the congruence kernels with respect to these actions coincide.
\end{mainthm}

\begin{restatable}{mainthm}{bkernel}\label{branchkernel}
 Let $G$ have two branch actions on trees. 
 Then the branch kernels with respect to these actions coincide.
\end{restatable}

The above immediately imply that the rigid kernels of a branch group with respect to any two branch actions are naturally isomorphic.

We will deduce these theorems from a more powerful observation,
namely that the congruence and branch topologies of a branch group $G$ can be defined in purely group-theoretic terms, 
with no reference to a branch action, using the \emph{structure graph} $\mathcal{B}$ of $G$.
This graph depends only on the subgroup structure of $G$ and is related to every tree on which $G$ acts as a branch group.
As we shall see in Section 3 (\autoref*{prop}), if $G$ acts on $T$ as a branch group then $T$ embeds $G$-equivariantly in $\mathcal{B}$, 
where the action of $G$ on $\mathcal{B}$ is that induced by conjugation on subgroups.
Further, the image of $T$ is coinitial in $\mathcal{B}$ (\autoref*{lemma}).
Once we have analogues of $\St_G(n)$ and $\rst_G(n)$ for the action of $G$ on $\mathcal{B}$, 
the above mentioned results are the key to showing that 
all congruence and branch topologies induced by a branch action coincide; they all agree with the topologies with respect to $\mathcal{B} $.

In Section \ref*{structure lattice} we define the structure graph and the larger \emph{structure lattice} of a branch group. 
These very useful objects were first introduced in \cite{wilsonJIclassification} and \cite{wilsonNewhorizons}
and used to analyse just infinite groups. 
They were also used in \cite{Hardy} to characterize branch groups in purely group-theoretic terms.
In those settings, they are defined as quotients of the lattice of subnormal subgroups of a branch group.
Here we give a more direct description by examining the subgroups with finitely many conjugates.

%As we shall see in Section 3, 
%all trees on which a branch group $G$ has a branch action can be embedded in an object defined by the subgroup structure of the branch group,
%its \emph{structure graph}.
%The conjugation action of $G$ on its subgroups induces an action on the structure graph and the embedding of the tree is $G$-equivariant. 
%Each such embedded tree is cofinal in the structure graph. 
%It therefore makes sense to define analogues of $\St_G(n)$ and $\rst_G(n)$ for the structure graph and use these to define the congruence and branch topologies.

%The structure graph and the larger \emph{structure lattice} of a branch group were first described in \cite{wilsonNewhorizons} and \cite{branchstructure}.
%We begin Section \ref{structure lattice} by giving a more natural definition of these objects \TODO{finish summary of paper}

%Waffle....
%Introduce congruence subgroup problem, links with Serre work.
%Talk about BSZ and what they do ``the preliminary question of great importance is...''
%State my results (they're much better).
%How it depends on subgroup structure of branch groups (structure lattice, etc)

\section{The structure lattice and structure graph}\label{structure lattice}

\medskip
\noindent\textbf{Notation.}  
We write $H\lf G$ and $H\nf G$ to indicate, respectively,
that $H$ is a finite index subgroup of $G$ and that $H$ is a normal finite index subgroup of $G$.
We also use the standard notation $\N(H)$ (resp.\ $\C(H)$) for the normalizer (resp.\ centralizer) of a subgroup $H$ in $G$. 
Furthermore, $H^G$ will denote the subgroup generated by all conjugates of $H$ by $G$.
Throughout the rest of the paper, $G$ will denote a branch group.

\medskip
Subgroups of branch groups are subject to several constraints.
The proof of \cite[Lemma 2]{geomded} shows that branch groups have no non-trivial virtually abelian normal subgroups 
and the following is obtained in \cite[Theorem 4]{slavachapter}:

\begin{thm}\label{slava}  
Suppose that $G$ is a branch group acting on a tree $T$ and let $K\trianglelefteq G$ with $K\neq1$. 
Then $K$ contains the derived subgroup $\rst_G(n)'$ of $\rst_G(n)$ for some integer $n$. 
\end{thm}

Thus all branch groups are just non-(virtually abelian); that is, they are not virtually abelian but all of their proper quotients are.

Let $L(G)$ be the collection of all subgroups of $G$ which have finitely many conjugates (in other words, whose normalizer has finite index).
If $H,K\in L(G)$, then clearly $H\cap K, \langle H,K\rangle \in L(G)$.
Thus $L(G)$ forms a lattice with respect to subgroup inclusion,
with $H \cap K$ and $\langle H, K \rangle$ respectively the meet and join of two elements $H,K$. 
Lemma 2.2 of \cite{mewil} shows that $L(G)$ contains no non-trivial virtually soluble subgroups. 
We will use this without further comment in the remainder of the paper.

This allows us to prove the following, which is a generalization of \cite[Theorem 8.3.1]{Hardy} and \cite[Lemma 4.3]{wilsonNewhorizons}.

\begin{prop}\label{centralizers}
 Let $H, K\in L(G)$ with $K\trianglelefteq H$ and $H/K$ virtually nilpotent. 
 Then $\C(K)=\C(H)$.
\end{prop}

\begin{proof}
First we claim that if $A$ is a subgroup with finitely many conjugates in a group $\Gamma$ and $A$ is virtually nilpotent then so is $ A^{\Gamma}$.
Let $N$ be the normal core of $\mathrm{N}_{\Gamma}(A)$, so that $A_0:= A \cap N$ is virtually nilpotent and normal in $N$.
By Fitting's theorem (\cite[5.2.8]{Robinson}), $A_0$ has a unique maximal nilpotent normal subgroup, $B$ say, 
which is normal in $N$.
The finitely many $\Gamma$-conjugates of $B$ are also nilpotent and normal in $N$;
thus  $B^{\Gamma}$ is nilpotent, again by Fitting's theorem.
It remains to show that $ B^{\Gamma}$ has finite index in $ A^{\Gamma}$. 
Since $A/B$ is finite so is the quotient $(A B^{\Gamma})/ B^{\Gamma}$
and this has finitely many conjugates in ${\Gamma}/B^{\Gamma}$ because $A$ has finitely many in $\Gamma$.
Therefore the quotient $(A^{\Gamma} B^{\Gamma})/B^{\Gamma} \cong A^{\Gamma}/B^{\Gamma}$ is finite 
by Dicman's lemma (see \cite[14.5.7]{Robinson}) and our claim is proved.

Suppose that $H,K\in L(G)$ with $H/K$ virtually nilpotent and write $C:=\C(K)\in L(G)$.
We will prove the proposition for $H_1:=H^C$, $K$ and deduce the result for $H,K$ from this.
To see that $H_1/K$ is virtually nilpotent, note that for each $c\in C$ the conjugate $(H/K)^c=H^c/K$ is isomorphic to $H/K$.
There are finitely many of these $C$-conjugates, as $H\in L(G)$, so it follows from the claim that $H_1/K$ is virtually nilpotent.
Now, $C\cap K\in L(G)$ is abelian, hence trivial, and we have
\setlength{\abovedisplayskip}{1pt}
\setlength{\belowdisplayskip}{1pt}
\begin{equation*}
C\cap H_1=(C\cap H_1)/(C\cap K) \cong K(C\cap H_1)/K \leq H_1/K. 
\end{equation*}
Thus  $C\cap H_1\in L(G)$ is virtually nilpotent and therefore trivial.
Note that $C$ is normalized by $H_1$, since $K\trianglelefteq H_1$, whence $[C,H_1]\leq C$. 
As $H_1$ is also normalized by $C$ we have $[C,H_1]\leq C\cap H_1=1$ and therefore $C\leq \C(H_1)$. 
The proof is complete as $K\leq H\leq H_1$ implies that $\C(H_1)\leq \C(H)\leq C$.
\end{proof}

\noindent\textbf{Notation.} For $H,K \in L(G)$, we write $K\lva H$ (respectively, $K\nva H$) if $K\leq H$ (resp.\ $K\trianglelefteq H$)
and $K$ contains the derived group of a finite index subgroup of $H$. 
Note that if $K\nva H$ then $H/K$ is virtually abelian.
\medskip

\autoref*{centralizers} has the following consequences.

\begin{lem}\label{trivial_commuting}  
Let $H_1,H_2\in L(G)$. Then $H_1\cap H_2=1$ if and only if $[H_1,H_2]=1$. 
\end{lem}
\begin{proof}
%\begin{enumerate}[label=(\roman*)]
 %\item Since $K\leq_{va} H$, we have $K\geq A'$ for some $A\leq_f H$. 
 %Let $N=\mathrm{Core}_H(A)$ be the normal core of $A$ in $H$. 
 %Then $N\trianglelefteq_f H$ and $N'\triangleleft L=\mathrm{Core}_H(K)\triangleleft H$; that is $L\trianglelefteq_{va} H$. 
 %Thus $\C(L)=\C(H)$, by \ref{centralizers}, but then $L\leq K\leq H$ implies that $\C(K)=\C(L)=\C(H)$, as required. 
 
  If $[H_1,H_2]=1$ then $H_1\cap H_2\in L(G)$ is abelian and therefore trivial.
 
 For the converse, let $N$ be the normal core of the intersection $\N(H_1)\cap\N(H_2)$.
 Thus $N\nf G$ normalizes $H_1$ and $H_2$.
 For $i=1,2$, let $K_i=H_i\cap N$. 
 Then $K_i\nf H_i$ and  $K_i\in L(G)$.
 Now, since $K_i\trianglelefteq K_1K_2$, we have 
 $[K_1,K_2]\leq K_1\cap K_2\leq H_1\cap H_2=1.$
 Therefore, applying  \autoref*{centralizers} to $K_i\lf H_i$,
 we obtain $\C(H_1)=\C(K_1)\leq \C(K_2)=\C(H_2)$, and vice-versa, so $[H_1,H_2]=1$. 
%\end{enumerate}
\end{proof}

%\begin{lem}\label{trivial_commuting}
% Let $H_1,H_2\inL(G)$. Then $H_1\cap H_2=1$ if and only if $[H_1,H_2]=1$. 
%\end{lem}
%\begin{proof}
% If $[H_1,H_2]=1$ then $H_1\cap H_2\in L(G)$ is abelian and therefore trivial by Lemma \ref{Hardy8.2}.
 
% Conversely, since $H_1,H_2\inL(G)$, their normalizers have finite index in $G$. 
% Write $N$ for the normal core of the intersection $\N(H_1)\cap\N(H_2)$, so $N\triangleleft_f G$ normalizes each $H_i$.
% For $i=1,2$, let $K_i=H_i\cap N\triangleleft_f H_i$ and note that $K_i\inL(G)$ as $N\inL(G)$.
% Now, since $K_i\triangleleft K_1K_2$, we have 
% $$[K_1,K_2]\leq K_1\cap K_2\leq H_1\cap H_2=1.$$
% Therefore, applying Proposition \ref{centralizers} to $K_i\leq_f H_i$, we obtain $\C(H_1)=\C(K_1)\leq \C(K_2)=\C(H_2)$, and viceversa, so $[H_1,H_2]=1$. 
%\end{proof}

\begin{lem}\label{pseudo-complements}
 For every $H\in L(G)$ we have $\langle H, \C(H)\rangle= H\times \C(H)\lva G$.
\end{lem}
\begin{proof}
 Write $C=\C(H)$ and note that $\langle  H, C \rangle=H\times C$ by \autoref*{trivial_commuting}.
 If the normal core $N$ of $H$ is non-trivial then $N\nva G$, so $H\lva G$ and hence $H\times C\lva G$.
 Suppose then that $N=1$ and 
 let $V\in L(G)$ be the intersection of a maximal number of conjugates of $H$ such that $1<V\leq H$. 
 If $W$ is a conjugate of $V$ which is not contained in $H$ then
 $1\leq W\cap H\leq H$ is the intersection of one more conjugate of $H$ than $V$.
 Therefore $W\cap H=1$, by the choice of $V$ and $W\leq C$ by  \autoref*{trivial_commuting}.
 This implies that $H\times C$ contains all conjugates of $V$; in particular, it contains their product $V^G \nva G$.
 Thus $H\times C\lva G$, as required.
 \end{proof}

\begin{lem}\label{vaiffcentralizers}
 Let $H, K\in L(G)$. The following are equivalent:
 \begin{enumerate}[leftmargin=2em,label=\textup{(\roman*)}]
  \item\label{H=K} $H\cap K\lva H, K;$
  \item\label{CH=CK} $\C(H)=\C(K);$
  \item\label{commonD} there exists $D\in L(G)$ such that $H\times D\lva G$ and $K\times D\lva G$.
 \end{enumerate}
\end{lem}
\begin{proof}
 If \ref*{H=K} holds, we have $A'\leq H\cap K$ for some $A\lf H$. 
 Let $N$ and $L$ be, respectively,  the normal cores of $A$ and $H\cap K$ in $H$. 
 Then $N\nf H$ and $N'\trianglelefteq L\trianglelefteq H$; that is $L\nva H$. 
 Thus $\C(L)=\C(H)$ by \autoref*{centralizers}, but then $L\leq K\leq H$ implies that $\C(H\cap K)=\C(L)=\C(H)$.
 Repeating the procedure with $H$ replaced by $K$ yields $\C(H\cap K)=\C(K)=\C(H)$.
 This immediately implies \ref*{commonD} by \autoref*{pseudo-complements}.

 Suppose that \ref*{commonD} is true; so $G_0'\leq H\times D$ for some $G_0\lf G$.
  Then $G_0\cap K\lf K$ and, since $D\cap K=1$, we have
 $ (G_0\cap K)'\leq (H\times D)\cap K = H\cap K;$
 that is, $H\cap K\lva K$. 
 The same argument with $H$ and $K$ swapped gives $H\cap K\lva H$.  
\end{proof}

\subsection*{The structure lattice}
%Now we construct the structure lattice and structure graph from $L(G)$.
For $H,K \in L(G)$, write $K\sim H$ if any of the equivalent conditions of \autoref*{vaiffcentralizers} holds.
It is immediate that $\sim$ is an equivalence relation on $L(G)$.
We show that $\sim$ is a congruence on $L(G)$.
\begin{prop}
  Let $H_1,H_2,K_1,K_2\in L(G)$ with $H_1\sim H_2$ and $K_1\sim K_2$.
  Then $H_1\cap K_1 \sim H_2 \cap K_2$ and $\langle H_1, K_1\rangle \sim \langle H_2,K_2\rangle$.
\end{prop}

\begin{proof}

 %We first show that $\sim$ respects the meet operation.
 For $i=1,2$ we have $H_1\cap H_2\cap K_i\lva H_i\cap K_i$ and $K_1\cap K_2\cap H_i\lva H_i\cap K_i$ so that
 %$(H_1\cap H_2)\cap (K_1\cap K_2)=
 $(H_1\cap H_2\cap K_i)\cap (K_1\cap K_2\cap H_i)\lva H_i\cap K_i.$
 Thus $H_1\cap H_2 \sim K_1\cap K_2$.
 
 To show that the join operation is respected, write $M:=\langle H_1\cap H_2, K_1\cap K_2\rangle \in L(G)$.
 Then $\C(M)$ centralizes $H_i$ and $K_i$ for $i=1,2$ so that \autoref*{trivial_commuting} yields $\C(M)\cap H_i=1$ and $\C(M)\cap K_i=1$.
 Therefore $\langle \C(M), H_i, K_i\rangle = \C(M)\times \langle H_i, K_i\rangle \geq \C(M)\times M$ and, 
 since $\C(M)\times M\lva G$ (by \autoref*{pseudo-complements}), 
 we have  $\C(M)\times \langle H_i, K_i\rangle \lva G$ for $i=1,2$.
 \end{proof}

By the above, the join and meet of two equivalence classes $[H],[K]$ of elements $H,K\in L(G)$ is well defined by
$$[H]\vee[K]=[\langle H, K\rangle] \quad \text{and} \quad [H]\wedge[K]=[H\cup K],$$
and the quotient $\mathcal{L}=L(G)/\sim$ is again a lattice with respect to these operations and the natural partial order inherited from $L(G)$:
$$[K]\leq [H] \quad \text{if and only if} \quad [K\cap H]=[K].$$

This quotient  $\mathcal{L}$ is the \emph{structure lattice} of $G$.

Note that $[G]$ and $[1]=\{1\}$ are, respectively, the greatest and least elements of $\mathcal{L}$.
Observe also that $[H]^g=[H^g]$ for each $H\in L(G)$ and $g\in G$.
Thus the action of $G$ on its subgroups by conjugation induces a well-defined action on $\mathcal{L}$.
It is shown in \cite{Hardy} that $\mathcal{L}$ is a Boolean lattice (it is uniquely complemented and distributive), but we do not require this fact here. 

\subsection*{The structure graph}
An element $B$ of $L(G)$ is \emph{basal} if $\langle B^G\rangle$ is the direct product  of the finitely many conjugates of $B$; 
in particular, if $B^g\neq B$ then $B^g\cap B=1$.
Examples of basal subgroups include $\rst_G(v)$ for every vertex $v$ of a tree on which $G$ acts as a branch group.

Basal subgroups have the following useful properties.
\begin{lem}[\cite{wilsonNewhorizons}]\label{basalprops}Let $B, B_1, B_2$ be basal subgroups of $G$. Then
 \begin{enumerate}[leftmargin=2em,label=\textup{(\roman*)}]
   \item $B_1\cap B_2$ is basal;
   \item if $[B_1]\leq [B_2]$ then  $\N(B_1)\leq \N(B_2)$;
   \item $\N(B)$ is the stabilizer of $[B]$ under the action of $G$ by conjugation;
   \item $\bigcap (\N(B)\mid B \text{ is basal})=1$.
 \end{enumerate}
\end{lem}

\begin{proof}
 \begin{enumerate}[fullwidth,itemsep=0pt,label=(\roman*)]
  \item Suppose that $(B_1\cap B_2)^g=B_1^g\cap B_2^g\neq B_1\cap B_2$ for some $g\in G$.
  Then either $B_1^g\neq B_1$ or $B_2^g\neq B_2$. 
  In each case we have $(B_1\cap B_2)^g\cap (B_1\cap B_2)=1$.
  \item Let $g\in \N(B_1)$. 
  Then $1\neq [B_1]^g=[B_1]\leq [B_2]^g\wedge[B_2]=[B_2^g\cap B_2]$ so $B_2^g=B_2$ and $g\in \N(B_2)$, as required.
  \item This follows from the argument in the previous part.
  \item Note that $\N([\rst_G(v)])=\St_G(v)$ for every $v\in T$, where $T$ is a tree on which $G$ acts as a branch group.
  The claim follows from the observation that every $\rst_G(v)$ is basal and $\bigcap(\St_G(v) \mid v\in T)=1$.
 \end{enumerate}
\end{proof}

The \emph{structure graph} $\mathcal{B}$ of $G$ has as vertices all non-trivial elements $[B]\in\mathcal{L}$ such that $B$ is basal.
Two elements $[A]\leq [B]$ of $\mathcal{B}$ are joined by an edge if $[A]$ is maximal subject to this inequality.
Again, the conjugation action of $G$ on its basal subgroups induces an action on $\mathcal{B}$ by graph automorphisms.

\section{The congruence and branch topologies}

\noindent\textbf{Notation.} Since in this section we will deal with different branch actions of the same group $G$, 
we will write $\St_{\rho}(n)$ and $\rst_{\rho}(v)$ for the stabilizer of the $n$th layer and the rigid stabilizer of vertex $v$
\emph{with respect to a given branch action $\rho:G\rightarrow \Aut(T)$}. 
We shall omit the subscript when there is no risk of confusion.

\begin{prop}[\cite{Hardy}]\label{prop}
 If $G$ acts as a branch group on a tree $T$ then there is an order-preserving $G$-equivariant embedding 
 $\phi: T \rightarrow \mathcal{B}$ defined by 
 $\phi: v\mapsto [\rst(v)]$.
\end{prop}
\begin{proof}
 We have already seen that $\phi$ is $G$-equivariant as
 $$\phi(v)^g=[\rst(v)]^g=[\rst(v)^g]=[\rst(v^g)]=\phi(v^g).$$ 
 That $\phi$ is order-preserving is also clear since $v$ is a descendant of $w$ if and only if $\rst(v)\leq \rst(w)$.

 To see that $\phi$ is injective, let $\rst(v)\sim\rst(w)$.
 If $v$ and $w$ were incomparable vertices, then $\rst(v)\cap \rst(w)=1$ would imply that  $\rst(v)$ and $\rst(w)$ are virtually abelian, a contradiction.
 Thus $v$ and $w$ are comparable; say $v\leq w$.
  Suppose for a contradiction that $v\neq w$;
 then there is a vertex $v_2\neq v$ in the same layer as $v$ such that $v_2\leq w$ (a `sibling' of $v$).
 Since $v,v_2 \leq w$ we have $\rst(v)\times \rst(v_2)\leq \rst(w)$ and $\rst(v)\lva\rst(w)$.
 It then follows that $\rst(v)\lva (\rst(v)\times \rst(v_2))$ and that $\rst(v_2)$ is virtually abelian. 
 This gives the desired contradiction, so $v=w$, as required. 
\end{proof}

\begin{prop}[\cite{Hardy}]\label{lemma}
Let $G$ act as a branch group on $T$ and let $1\neq B\in L(G)$. 
Then there exists some $v\in T$ such that $[\rst(v)]\leq [B]$.
\end{prop}

\begin{proof}
 Since $B$ is non-trivial it contains a non-trivial element $g$ which moves some vertex $v$. 
 Let $w=v^g$ and $N$ be the normal core of the normalizer of $B$.
 Then $\rst(v)\cap N\nf \rst(v)$.
 Let $h,k\in \rst(v)\cap N$. 
 Note that $h^g\in \rst(v^g)=\rst(w)$; so $h^g$ and $k$ commute (because $[\rst(w),\rst(v)]=1$).
 Then  $[[h,g],k]\in B$ since $h,k$ normalize $B$ and we have
 $$[h,k]=h^{-1}k^{-1}hk=h^{-1}h^g k^{-1}(h^{-1})^g hk=[h^{-1}h^g,k]=[[h,g],k].$$
 Thus $(\rst(v)\cap N)'\leq B$ and $(\rst(v)\cap N)'\sim \rst(v)$ yields the result.
\end{proof}

The above imply that the structure graph ``contains'' all possible branch actions of $G$.
It is then reasonable to define the congruence and branch topologies with respect to the action of $G$ on $\mathcal{B}$.

When the structure graph is itself a tree, 
then  $G$ acts on it with a branch action 
and all other trees on which $G$ acts as a branch group are obtained from the structure graph by ``deleting layers''.
This was proved in \cite{geomded}, where a sufficient condition for the structure graph to be a tree is given.
Examples of branch groups satisfying that condition include Grigorchuk's first group, the Gupta--Sidki $p$-groups and the Hanoi tower group.
A necessary and sufficient condition for the structure graph to be a tree is given in \cite{Hardy}.

\subsection*{The congruence topology}

In order to define the congruence topology with respect to the action of $G$ on $\mathcal{B}$ we must find analogues of the level stabilizers $\St_G(n)$.
As $G$ acts level-transitively, the obvious analogue in $\mathcal{B}$ of a layer is an orbit $b^G$ where $b=[B]\in\mathcal{B}$ and $B$ is basal.
Recall from \autoref*{basalprops} that $\St_G([B])=\N(B)$, so the analogue of a level stabilizer $\St_G(n)$ is an \emph{orbit stabilizer}
$$\St_{\mathcal{B}}(b^G):=\bigcap (\N(B^g) \mid g\in G )$$
for  $b=[B]\in\mathcal{B}$.
Since each basal subgroup has only finitely many conjugates, these orbit stabilizers have finite index in $G$.
Furthermore, the intersection of all orbit stabilizers is trivial.
We take the family
$\{ \St_{\mathcal{B}}(b^G) \mid b\in \mathcal{B}\}$ 
as a neighbourhood basis of the identity to define the \emph{congruence topology} of $G$ with respect to the action of $G$ on $\mathcal{B}$.
The completion of $G$ with respect to this topology is a profinite group $\overline{G}_{\mathcal{B}}$ onto which the profinite completion $\widehat{G}$ maps.
We denote the kernel of this map by $C_{\mathcal{B}}$. 
In the following theorem we prove that for any branch action $\rho$ of $G$, 
the topologies induced by $\{\St_{\rho}(n)\mid n\geq 0\}$ and by $\{\St_{\mathcal{B}}(b^G)\mid b\in \mathcal{B}\}$ are equal.
Thus the congruence kernels with respect to each of them coincide and \autoref*{congruence} as stated in the introduction follows.

%\begingroup
%\def\themainthm{\ref*{congruence}}
%\begin{mainthm}
% Let $G$ have two branch actions $\rho, \sigma$ on trees $T_{\rho}, T_{\sigma}$. 
%The congruence kernels with respect to these actions are equal. 
%\end{mainthm}
%\addtocounter{mainthm}{-1}
%\endgroup

\begin{thm}
 For any branch action $\rho:G\rightarrow \Aut(T)$, denote by $C_{\rho}$ the congruence kernel with respect to this action.
 Then $C_{\rho}=C_{\mathcal{B}}$.
\end{thm}

\begin{proof}
 We must show that for every $n\geq 0$ there is some $b\in \mathcal{B}$ such that $\St_{\rho}(n)\geq\St_{\mathcal{B}}(b^G)$
 and conversely, that for every $b\in \mathcal{B}$ there is some $n\geq 0$ such that $\St_{\mathcal{B}}(b^G)\geq \St_{\rho}(n)$.
 However, by \autoref*{prop}, the $n$th layer $V_n$ of the tree corresponds to an orbit of basal subgroups so that the former statement holds trivially.
 It therefore suffices to show the latter statement.
 
 For a given $b=[B]\in \mathcal{B}$, \autoref*{lemma} gives a vertex $v$ of $T$ such that $[\rst_{\rho}(v)]\leq b$.
 Suppose that $v$ is in the $n$th layer of $T$ and let $x\in \St_{\rho}(n)$. 
 Note that $\St_{\rho}(n)$ is the pointwise stabilizer of the $G$-orbit of $[\rst_{\rho}(v)]$ by  \autoref*{prop}.
 Then 
 \setlength{\abovedisplayskip}{1pt}
\setlength{\belowdisplayskip}{1pt}
 $$1\neq [\rst_{\rho}(v^x)]=[\rst_{\rho}(v)]\leq b^x \wedge b.$$
 Since $B$ is basal this implies that $B^x=B$ (so $b^x=b$).
 Thus the normal subgroup $\St_{\rho}(n)$ fixes all elements of $b^G$ and we conclude that $\St_{\rho}(n)\leq \St_{\mathcal{B}}(b^G)$, as required.
\end{proof}

\subsection*{The branch topology}
As with the congruence kernel, \autoref*{branchkernel} will follow from the fact that
the branch kernel with respect to a branch action is equal to the analogous object for the action on the structure graph. 
We must therefore define the branch topology with respect to $\mathcal{B}$.
To do this we generalize the notion of a rigid stabilizer.
%\begin{defn}
 For a non-trivial basal subgroup $A$ of a branch group, define its \emph{rigid normalizer} $\R(A)$ by
 $$\R(A):=\bigcap (\N(B) \mid B \text{ is basal and } A\cap B=1).$$
%\end{defn}

We will use the following properties of rigid normalizers. 
These were proved in \cite{Hardy} but we include a shorter argument for the reader's convenience.

\begin{prop}\label{rigid}
Let $A, A_1, A_2$ be basal subgroups and $v$ a vertex of a tree on which $G$ acts as a branch group. Then
 \begin{enumerate}[leftmargin=2em,label=\textup{(\roman*)}]
  \item\label{A<R(A)} $A\leq \R(A)\leq \N(A);$
  \item\label{equalR} if $[A_1]\leq [A_2]$ then $\R(A_1)\leq \R(A_2)$ 
  \textup{(}in particular, $\R(A_1)=\R(A_2)$ if $[A_1]=[A_2]);$
  \item $\R(A)$ is basal and the unique maximal element of $[A]$;
  \item\label{Rrst} $\R(\rst(v))=\rst(v)$.
 \end{enumerate}
\end{prop}

\begin{proof}
 \begin{enumerate}[fullwidth,label=(\roman*),topsep=0pt]
  \item Let $B$ be a basal subgroup with $A\cap B=1$.
  Then $[A,B]=1$ by \autoref*{trivial_commuting}
  and so $A\leq \N(B)$; thus $A\leq \R(A)$.
    If $g\in\R(A)$ then $g$ normalizes each of the conjugates of $A$
    distinct from $A$ as they are basal and have trivial intersection with $A$. 
  Therefore $g$ must normalize $A$ itself.

  \item Since $[A_1]\leq[A_2]$, there is some finite index subgroup $K$ of $A_1$ such that $K'\leq A_1\cap A_2$.
  Now, for a basal subgroup $B$ with $B\cap A_2=1$, we have
  $$(K\cap B)'\leq K'\cap B\leq A_1\cap A_2\cap B\leq A_2\cap B=1$$
  and $B\cap K\lf B\cap A_1$.
  That is to say, $B\cap A_1$ is virtually abelian and so $B\cap A_1=1$.
  Hence every $g\in \R(A_1)$ normalizes $B$ and $\R(A_1)\leq \R(A_2)$, as required.

  \item We first show that $\R(A)\sim A$ using \autoref*{vaiffcentralizers}\ref*{commonD}.
   Let $D$ denote the product of the finitely many conjugates of $A$ that are distinct from $A$; 
  then $A^G=A\times D\nva G$.
  Since $A\leq \R(A)$, it suffices to show that $\R(A)\cap D=1$.
  Let $x\in \R(A)\cap D$ and $B$ be a basal subgroup. 
  If $B\cap A=1$ then $x\in \R(A)$ normalizes $B$.
  If $B\cap A\neq 1$ then $B\cap A\leq A$ is basal and is centralized by $x\in D\leq \C(A)$; 
  hence $x\in \N(B\cap A)\leq \N(B)$ by \autoref*{basalprops}.
  Thus $x\in \bigcap (\N(B) \mid B \text{ is basal })=1$.
   
  To see that $\R(A)$ is basal, suppose that $\R(A)^g\neq\R(A)$ for some $g\in G$.
  Then, as $\R(A)^g=\R(A^g)$, we have $A\neq A^g$ by \ref*{equalR} so that $A^g\cap A=1$ and $[A^g,A]=1$ by \autoref*{trivial_commuting}.
  Now, $\R(A)\sim A$, so $\C(\R(A))=\C(A)$ by \autoref*{vaiffcentralizers}  which implies that $[A^g,\R(A)]=1$. 
  Similarly, $\R(A)^g\sim A^g$ implies that $[\R(A)^g,\R(A)]=1$, yielding $\R(A)^g\cap \R(A)=1$.
  
  For any $H\in[A]$ and any basal subgroup $B$ such that $B\cap A=1$ we have $[B,A]=1$, so $B\leq \C(A)=\C(H)$.
  This means that $[B,H]=1$, in particular, $H\leq \N(B)$. 
  Hence $H\leq \R(A)$.

  \item By part \ref*{A<R(A)}, it suffices to show that $\R(\rst(v))\leq \rst(v)$. 
  Let $g\in \R(\rst(v))$ and $w\in T\setminus T_v$.
  If $w$ is incomparable with $v$ then $\rst(w)\cap \rst(v)=1$ and so $g$ fixes $w$.
  If $v\leq w$ then $\rst(v)\leq \rst(w)$ and $\R(\rst(v))\leq\N(\rst(v))\leq\N(\rst(w))$ by \autoref*{basalprops}. 
  Thus $w=w^g$ and the claim follows.
 \end{enumerate}
\end{proof}

By analogy with the rigid stabilizers of layers, 
the \emph{rigid normalizer of an orbit}  ${a^G=[A]^G}$ in $\mathcal{B}$ 
is defined to be $\R(a^G):=\R(A)^G$.
That these rigid normalizers of orbits have finite index in $G$ follows from \autoref*{lemma}.
It also follows that the intersection of all of them is trivial.
We may therefore take $\{\R(a^G)\mid a\in \mathcal{B} \}$ as a neighbourhood basis of the identity 
to generate the \emph{branch topology} of $G$ with respect to the action of $G$ on $\mathcal{B}$.
The completion $\widetilde{G}_{\mathcal{B}}$ with respect to this topology is a profinite group in which $G$ embeds.
We denote by $B_{\mathcal{B}}$ the kernel of the map $\widehat{G}\rightarrow \widetilde{G}_{\mathcal{B}}$.
In our last theorem we prove that for any branch action $\rho$ of $G$, 
the topologies induced by $\{\rst_{\rho}(n)\mid n\geq 0\}$ and by $\{\R(a^G) \mid a \in \mathcal{B}\}$ are equal;
consequently, the branch kernels with respect to each of them coincide and \autoref*{branchkernel} as stated in the introduction follows.

\begin{thm}
 For a branch action $\rho: G \rightarrow \Aut(T)$, denote by $B_{\rho}$ the branch kernel with respect to this action. 
 Then $B_{\rho}=B_{\mathcal{B}}$.
\end{thm}

\begin{proof}
 As with the proof for the congruence kernels, we must show that
 for every $n\geq 0$ there is some $b\in \mathcal{B}$ such that $\rst_{\rho}(n)\geq\ \R(b^G)$
 and  that for every $b\in \mathcal{B}$ there is some $n\geq 0$ such that $\R(b^G)\geq \rst_{\rho}(n)$.
 The former claim follows from \autoref*{rigid}\ref*{Rrst} since all rigid stabilizers of vertices are basal.
 It therefore suffices to prove the latter claim.
 Given $b\in\mathcal{B}$, \autoref*{lemma} yields that $[\rst_{\rho}(v)]\leq b$ for some $v\in T$.
 Suppose that $v\in V_n\subset T$.
 We then have
 $\rst_{\rho}(v)=\R(\rst_{\rho}(v))\leq \R(b)$ by \autoref*{rigid}, 
 and the transitivity of $G$ on each of the layers of $T$ 
 implies that $\rst_{\rho}(n)\leq \R(b^G)$, as required.
\end{proof}

%\begingroup
%\def\themainthm{\ref*{branchkernel}}
%\begin{mainthm}
% Let $G$ have two branch actions $\rho, \sigma$ on trees $T_{\rho}, T_{\sigma}$. 
% The branch kernels with respect to these actions are equal. 
%\end{mainthm}
%\addtocounter{mainthm}{-1}
%\endgroup

%The following is now immediate, as the congruence topology and branch topology are independent of the branch action of $G$ on any given tree.

%\begingroup
%\def\themaincor{\ref*{rigidkernel}}
%\begin{maincor}
% Let $G$ have two branch actions $\rho, \sigma$. 
% The rigid kernels $\ker(\widetilde{G}\to \overline{G})$ with respect to these actions are equal. 
%\end{maincor}
%addtocounter{maincor}{-1}
%\endgroup

\bibliographystyle{abbrv}
\bibliography{cspindependent}

\end{document}